\documentclass[a4paper]{amsart}
\usepackage{amssymb}
\usepackage{amsmath}

\usepackage{amsfonts}
\usepackage{pdfsync}
\usepackage{color}
\usepackage{graphicx}
\usepackage{epsf}
\usepackage{eucal}
\usepackage{amscd}
\usepackage{fullpage}
\usepackage{paralist}
\usepackage[shortlabels]{enumitem}
\usepackage[all]{xy}
\usepackage{tikz}
\usepackage{verbatim}
\usepackage{cite}
\usepackage{mathrsfs}

\usepackage{hyperref}
\usepackage{cleveref}

\hypersetup{colorlinks=false}

\usetikzlibrary{calc}

\usepackage{color}
\definecolor{darkgreen}{cmyk}{1,0,1,.2}
\definecolor{m}{rgb}{1,0.1,1}

\newdimen\theight
\def\TeXref#1{%
             \leavevmode\vadjust{\setbox0=\hbox{{\tt
                     \quad\quad  {\small \textrm #1}}}%
             \theight=\ht0
             \advance\theight by \lineskip
             \kern -\theight \vbox to
             \theight{\rightline{\rlap{\box0}}%
             \vss}%
             }}%

\theoremstyle{plain}

\newtheorem{thm}{Theorem}
\newtheorem{lem}[thm]{Lemma}
\newtheorem{cor}[thm]{Corollary}
\newtheorem{prop}[thm]{Proposition}

\theoremstyle{definition}

\newtheorem{ex}[thm]{Example}

\theoremstyle{remark}

\newtheorem{rem}{Remark}

\newcommand{\ZZ}{{\mathbb  Z}}

\newcommand{\NN}{{\mathbb N}}

\newcommand{\GG}{{\mathcal G}}

\newcommand{\gromov}{\widehat{\GG}_*}

\newcommand{\Aut}{\mathop{\rm Aut}}
\newcommand{\id}{\mathop{\rm id}}

\DeclareMathOperator{\im}{im}
\DeclareMathOperator{\dom}{dom}

\numberwithin{equation}{section}

\newlist{propenum}{enumerate}{1} 
\setlist[propenum]{label=\emph{(\alph*)}, ref=\theprop~(\alph*)}
\crefalias{propenumi}{prop}

\newlist{defnenum}{enumerate}{1} 
\setlist[defnenum]{label=\emph{(\alph*)}, ref=\thedefn~(\alph*)}
\crefalias{defnenumi}{defn}

\newlist{lemenum}{enumerate}{1} 
\setlist[lemenum]{label=\emph{(\alph*)}, ref=\thelem~(\alph*)}
\crefalias{lemenumi}{lem}

\newlist{thmenum}{enumerate}{1} 
\setlist[thmenum]{label=\emph{(\alph*)}, ref=\thethm~(\alph*)}
\crefalias{thmenumi}{thm}

\title{Genericity of chaos for colored graphs}

\author[R. Barral Lij\'o]{Ram\'on Barral Lij\'o}
\address{
	Research Organization of Science and Technology\\
	Ritsumeikan University\\
	Nojihigashi 1-1-1, Kusatsu, Shiga, 525-8577, Japan}
\email{ramonbarrallijo@gmail.com}

\author[Hiraku Nozawa]{Hiraku Nozawa}
\address{Department of Mathematical Sciences\\
	Colleges of Science and Engineering\\
	Ritsumeikan University\\
	Nojihigashi 1-1-1, Kusatsu, Shiga, 525-8577, Japan}
\email{hnozawa@fc.ritsumei.ac.jp}


\date{\today}

\begin{document}

\maketitle

\begin{abstract} 
    To each colored graph, one can associate its closure in the universal space of isomorphism classes of pointed colored graphs, and this subspace can be regarded as a generalized subshift. Based on this correspondence, we extend the notion of chaotic dynamical systems to colored graphs. We introduce definitions for chaotic and almost chaotic (colored) graphs, and prove their topological genericity in various subsets of the universal space. 
\end{abstract}

\section{Introduction}

\emph{Chaotic dynamical systems} are one of central objects in the study of modern mathematics.
The aim of this paper is to study 
an analogous notion
in a setting that generalizes both symbolic dynamics and one-dimensional foliated spaces: the universal space of pointed colored graphs $\gromov$.  In what follows, we offer a very brief account on both of these mathematical fields.

Given a countable group $G$ and a finite set $F$ of colors, $G$ acts naturally by the left on the Cantor space $F^G$ of $F$-valued colorings on $G$ by the formula $(g\cdot \phi)(g') = \phi \big(g^{-1}g'\big)$. A closed, saturated subset of $F^G$ is called a \emph{subshift}, and they constitute the main subject of study in symbolic dynamics~\cite{LindMarcus}.
For every $\phi\colon G\to F$, the closure of its orbit under the $G$-action is a subshift.
In symbolic dynamics, properties similar to those that characterize chaos in a continuous setting have been studied for many years, e.g.~\cite{Ceccherini}. It was shown in~\cite{Fiorenzi} that, for $F$ finite, $F^G$ has density of periodic configurations if and only if $G$ is residually finite.

A \emph{foliated space} or \emph{lamination} is a topological space endowed with a partition into connected objects that locally are stacked together as a product. Usually, we require the elements of the partition (the \emph{leaves}) to be connected manifolds of a fixed dimension. However, we may also consider \emph{laminations by graphs} (also called \emph{graph matchbox manifolds}, see~\cite{Lukina}).

It is well-known that foliated spaces can be regarded as dynamical objects where the leaves play the role that  in a classical dynamical system would belong to the orbits. This is usually illustrated by the following picture (see e.g.~\cite[Chapter~2]{CandelConlonI}): Consider a path $\sigma\colon [0,1]\to X$ in the foliated space that  is contained in a single leaf and with starting point $\sigma(0)=x$. Then the local product structure enforces points close to $x$ to  follows paths that ``shadow" $\sigma$. The ambivalence implicit in the previous portrayal (e.g.\ the choice of ``shadowing paths'') already indicates that the dynamical representation of our foliated space (the \emph{holonomy pseudogroup}) depends on various choices, see \cite[Prop.~2.2.6]{CandelConlonI}. It is therefore far from unique. Nevertheless, there is a suitable notion of equivalence, originally introduced by Haefliger \cite{Haefliger} for foliated manifolds, that encompasses all such representations of the same foliated space.

Both laminations by graphs and symbolic dynamical systems admit a common generalization: the \emph{Gromov space of pointed colored graphs}, denoted by $\gromov$. We use this terminology because of its similarity with the Gromov space of pointed metric spaces,  see~\cite{BuragoBuragoIvanov}. As a set, $\gromov$ consists of pointed isomorphis classes $[X,x,\phi]$ of triples $(X,x,\phi)$, where $X$ is a connected graph with finite vertex degrees, $x\in X$ is a distinguished point, and $\phi\colon X \to \Xi$ is a (vertex-)coloring. For the purposes of this paper, we assume that $\Xi$ is a Cantor space. By an isomorphism $h\colon (X,x,\phi)\to (Y,y,\psi)$ we mean a graph isomorphism $h\colon X\to Y$ such that $h(x)=y$ and $\phi=\psi\circ h$.
In \cite{AldousLyons} the reader can enjoy a nice survey on both properties of $\gromov$ and open problems in the subject.

Associated to every colored graph $(X,\phi)$ there is a canonical map $\ell_{X,\phi}\colon X\to \gromov$, defined by the formula $x\mapsto [X,x,\phi]$. Let $[X,\phi]$ denote the isomorphism class of the colored graph $(X,\phi)$. We have $\im(\ell_{X,\phi})=\im(\ell_{Y,\psi})$ if $[X,\phi]=[Y,\psi]$, and $\im(\ell_{X,\phi})\cap\im(\ell_{Y,\psi})=\emptyset$ otherwise. By abuse of notation, we may identify the subset $\im(\ell_{X,\phi})$ and the equivalence class $[X,\phi]$. Here $[X,\phi]$ can be identified with $X/\Aut(X,\phi)$, which has a canonical graph structure. Therefore $\gromov$ is a space partitioned into graphs where each ``leaf" corresponds to an equivalence class $[X,\phi]$; it can be regarded as a singular lamination by graphs.
This is analogous to the ``smooth'' case, where one considers a Gromov space of Riemannian manifolds (see~\cite{AbertBiringer-unimodular,AlvarezBarralCandel2016,AlvarezBarral2017}). 

By the preceeding discussion,  one can associate to each colored graph $[X,\phi]$ the (generalized) dynamical system $\overline{[X,\phi]}\subset \gromov$. Here, we consider a natural topological structure of $\gromov$ given by the following notion of convergence: let  $[X,x,\phi]$ be a pointed, colored graph, and let $D_X(x,r)$ denote the \emph{disk} or \emph{closed ball} of center $x$ and radius $r$. Then $[X,x,\phi]$ is the limit of the sequence $[X_n,x_n,\phi_n]$ if, for every $r\in \mathbb{N}$ and $\epsilon>0$, there are graph isomorphisms 
\[
h_n\colon D_{X_n}(x_n,r ) \to D_{X}(x_,r )\;,
\]
defined for $n$ large enough, so that $h_n(x_n)=x$ and $d(\phi_n(y),\phi(h_n(y)))\leq \epsilon$ for all $v\in D_{X_n}(x_n,r)$.
This is in the same spirit as e.g.\ tiling theory, where in order to study a particular tiling one considers its continuous hull~\cite{ForrestHuntonKellendonk}. Then we may think of dynamical properties of $\overline{[X,\phi]}$ as attributes of $[X,\phi]$, since of course the latter determines the former. All these considerations guide us to the following principle: we shall say that $[X,\phi]$ is a $\emph{chaotic}$ colored graph if $\overline{[X,\phi]}\subset \gromov$ is chaotic as a generalized dynamical system, the precise description of which will be made explicit shortly.

Our setting is a generalization of symbolic dynamics in the following sense: If we fix the underlying graph to be the Cayley graph of a countable group $G$ with a system of generators,  then we can canonically identify each class $[G,g,\phi]$ with $[G,e_G,\phi\circ R_{g^{-1}}]$. Note that we are considering Cayley graphs with unlabeled edges. If we restrict our attention to graph isomorphisms given by translations (instead of the full isomorphism group of the Cayley graph), we recover the notion of subshift. 

One can define analogously a space $\mathcal{G}_*$ of isomorphism classes of pointed graphs $[X,x]$ with canonical injections $\ell_X\colon X\to \mathcal{G}_*$. All the preceeding considerations apply in a similar manner.

The usual definition of a chaotic dynamical system involves three conditions: sensitivity on initial conditions, topological transitivity, and density of closed orbits~\cite{Devaney}. The first one, requiring sensitivity on initial conditions, is usually formulated in terms of a Lyapunov exponent~\cite{Wolf}. 
Usually one needs to fix a harmonic measure on foliated spaces to define a Lyapunov exponent (see e.g. \cite{Deroin,Matsumoto}). Since we focus on a purely topological aspect of chaos in this article,
we will  ignore this condition  in our definition of chaos, cf.~\cite{Cairn}. Following the heuristic rule that orbits should correspond to leaves, it is natural to consider the following definitions: a colored graph $[X,\phi]$ is \emph{quasi-transitive} if $\im (\ell_{X,\phi})$ is a finite set, $[X,\phi]$ is \emph{(topologically) almost chaotic} if the subset of quasi-transitive classes is dense in $\overline{[X,\phi]}$. Note that we say almost chaotic, since in principle this would also encompass degenerate situations such as spaces consisting of a single compact leaf.  We say that $[X,\phi]$ is \emph{(topologically) chaotic} if it is almost chaotic, not quasi-transitive and \emph{aperiodic}, i.e.\  $\Aut(X,\phi)=\{\id\}$. 

The concepts that we have introduced thus far can be illustrated by the following simple examples,  the first of which is inspired by Champernowne's number \cite{Champernowne}.

\begin{ex}\label{ex.01}
Consider the lexicographical order on the set of finite sequences of $0$s and $1$s:
\[
0 < 1 < 00 < 01 < 11 < 000 < 001 < \cdots\;.
\]
We concatenate these sequences in order to construct the infinite sequence
\[
a_1 a_2 a_3 a_4 a_5 a_6 \cdots =01000111000001 \cdots\;.
\]
Let $\phi$ be the coloring of $\mathbb{Z}$ with two colors, $0$ and $1$, defined by
\[
\phi(n)=\begin{cases}
\;0 &\text{for } n<0\;,\\
\;a_n &\text{for } n\geq 0\;.
\end{cases}
\]
Then $\phi$ is a chaotic coloring of the integers. Indeed, since the non-positive integers determine the only infinite ray colored by $0$ and adjacent to a $1$, it is clear that $(\ZZ, \phi)$ is aperiodic. To show that $(\ZZ, \phi)$ is almost chaotic it is enough to note that, for every finite word $w$ with values in $\{0,1\}$, every concatenation $w^n$ will eventually appear on $a_n$.
\end{ex}

\begin{figure}[tb]
	\includegraphics{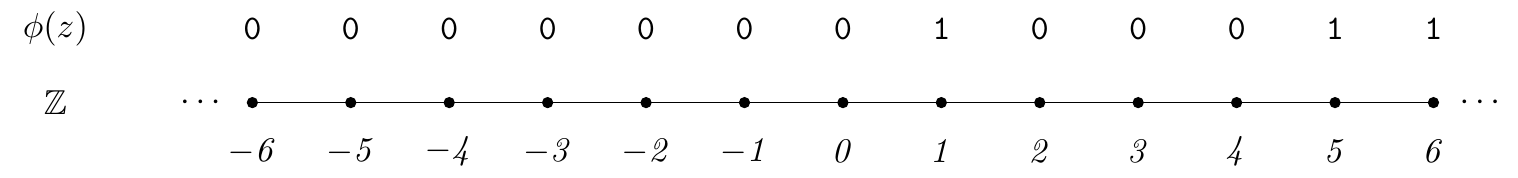}
	\caption{An illustration of \Cref{ex.01}}
\end{figure}

From the last example $(\ZZ, \phi)$, we can construct a chaotic graph $X$ without coloring. Indeed, such graph $X$ is obtained by adding a free edge to each vertex of $\ZZ$ colored with $1$ and by forgetting the coloring $\phi$.

The following is a straightforward generalization of \Cref{ex.01}.

\begin{ex}
Let $C\subset \Xi$ be a countable subset  with at least two elements. Let 
\[
\mathcal{W}=\{ c_{i_{1}} c_{i_{2}} \cdots c_{i_{k}} \mid c_{i_{j}} \in C\}
\]
be the set of finite words over the alphabet $C$. Fix a well-order $\mathcal{W}=\{w_n\}_{n\in\mathbb{N}}$  and take $c_{0} \in C$. Now we concatenate the elements of $\mathcal{W}$ using the well-order to define a coloring $\phi$ on $\ZZ$. Namely,
\[
\phi(n)=\begin{cases}
\, c_0 &\text{for }n<0\;,\\
\, w_m(i) &\text{for } n=i+\sum\nolimits_{l<m}|w_l|,\, 1\leq i\leq |w_m|\;;
\end{cases}
\]
where $|w|$ denotes the length of the corresponding finite word. The proof that $\phi$ is chaotic proceeds as in \Cref{ex.01}.
\end{ex}

The purpose of this paper is to study the topological genericity of the aforementioned notion of chaos for several subsets of $\gromov$. We will use the following notation. Let $\Delta\in \mathbb{N}$, and let $\gromov(\Delta)$ denote the subset of classes whose underlying graph $X$ satisfies $\deg X\leq \Delta$, and similarly for  $\mathcal{G}_*(\Delta)$. Also, we add $\infty$ as a superscript to denote the subset of classes whose underlying graph is infinite.

Recall that a property is topologically generic if it holds on a \emph{residual subset}, where a subset is residual if it contains a countable intersection of open dense sets. This notion is well-behaved for  \emph{Baire spaces}. \emph{Polish spaces}, i.e.\ separable and completely metrizable spaces, are all Baire spaces by the Baire Category Theorem. The following theorem is  the main result of the paper.

\begin{thm}\label{t.one}
	The following subsets are generic in the respective ambient spaces:
	\begin{thmenum}
	\item \label{i.ac} the  almost chaotic classes in $\mathcal{G}_*^\infty$, $\mathcal{G}_*^\infty(\Delta)$, $\gromov^\infty$ and $\gromov^\infty(\Delta)$, for $\Delta\geq 3$;
	\item \label{i.ap} the  aperiodic classes in  $\gromov^\infty$, $\gromov^\infty(\Delta)$; and
	\item \label{i.ch} the chaotic classes in $\gromov^\infty$ and $\gromov^\infty(\Delta)$ for $\Delta\geq 3$.
	\end{thmenum}
\end{thm}

All the subspaces appearing in Theorem~\ref{t.one} are Polish by \Cref{l.polish}, and in fact $\mathcal{G}_*^\infty(\Delta)$ and $\gromov^\infty(\Delta)$ are Cantor spaces. The reason for imposing $\Delta\geq 3$ is the following: For $\Delta=1$ we have $\gromov^\infty(1)=\emptyset$, so that the corresponding result is vacuously true. For $\Delta=2$, all connected infinite graphs are isomorphic to  either $\mathbb{N}$ or  $\ZZ$ with the obvious graph structures. Since the graph $\mathbb{N}$ is obviously aperiodic, the quasi-transitive classes cannot be dense in either $\gromov(2)$ nor $\mathcal{G}_*(2)$. Therefore, \Cref{i.ac} is false for $\Delta=2$. 

\begin{rem}
It is common to see in the literature a variation of $\gromov$ consisting of classes of pointed graphs endowed with both a vertex coloring and an edge coloring, see e.g.~\cite{AldousLyons}. The definitions of quasi-transitive, aperiodic, almost chaotic and chaotic classes can be obviously translated to this setting, and the same applies to the statement of \Cref{t.one}. The proof presented hereafter can be adapted without any difficulty. However, it adds significantly to the cumbersomeness of the notation, so the task is left to the interested reader.
\end{rem}

Finally, we present an example of a colored graph whose class is dense in $\gromov^{\infty}$. 

\begin{ex}\label{ex:dense}
Let $C\subset \Xi$ be a countable, dense subset.  Consider the set of  isomorphism classes of pointed, colored, finite and aperiodic graphs taking values in $C$, and let $P$ be a set of pointed, colored graphs containing exactly one representative of each such class. Since $P$ is countable, we can choose an enumeration of the form $P=\{ (Y_{i},y_{i},\phi_{i}) \}_{i \in \mathbb{Z}}$. Let $\psi$ be a aperiodic coloring of the integers taking values in $C$, and let $X$ be the colored graph constucted as follows: take the disjoint union of $(\ZZ,\psi)$ and the family $\{ (Y_{i},y_{i},\phi_{i}) \}_{i \in \mathbb{Z}}$, and add an edge joining $y_{i}$ to $i \in \mathbb{Z}$ for every $i\in \ZZ$.  

It is easy to see that $(X,\phi)$ is aperiodic: Note that $\mathbb{Z}$ is the unique bi-infinite simple path in $X$ and therefore it is preserved by any automorphism $f$ of $X$. Since $\phi|_{\mathbb{Z}}=\psi$ is aperiodic, $f$ fixes each vertex of $\ZZ$. So f preserves $(Y_i,y_{i})$ for each $i$. Since $(Y_i,y_{i},\phi_{i})$ is aperiodic, it follows that so is $(X,\phi)$. By construction, $(X,\phi)$ has any aperiodic colored finite graph as a subgraph, from which we conclude that $\overline{[X]}=\gromov^\infty$. Therefore, $(X,\phi)$ is chaotic if and only if the quasi-transitive classes are dense in $\gromov^\infty$. This can be proved directly by using a construction similar to that illustrated in Figure~\ref{f.ch}. Since the result also follows trivially from \Cref{i.ac}, we leave the details of the proof to the interested reader.
\end{ex}

\section{Preliminaries on graphs and colorings}\label{s: prelims on graphs & colorings}

\subsection{Graphs}\label{ss: graphs}

An ({\em undirected\/}) {\em graph\/} $X\equiv(X,E)$ consists of a set $X$ and a family $E$ of subsets $e\subset X$ with\footnote{The cardinality of a set $X$ is denoted by $|X|$.} $|e|=2$. The elements of $X$ and $E$ are called {\em vertices\/} and {\em edges\/}, respectively. We will identify a graph and its vertex set when no confusion may arise. If an edge $e$ contains a vertex $x$, it is said that $e$  and $x$ are {\em incident\/}. The {\em degree\/} $\deg x$ of a vertex $x$ is the number of edges incident to $x$. The \emph{degree} of $X$ is $\deg X=\sup_{x\in X}\deg x$. Two different vertices are {\em adjacent\/} if they define an edge. Two different edges are {\em consecutive\/} if they have a common vertex. For\footnote{We assume that $0\in\mathbb{N}$.} $n\in\mathbb{N}$, a {\em path\/} of {\em length\/} $n$ from $x$ to $y$ in $X$ is a sequence of $n$ consecutive edges joining $x$ to $y$; in terms of their vertices, it can be considered as a sequence $(z_0,\dots,z_n)$, where $z_0=x$, $z_n=y$, and $z_{i-1}$ and $z_i$ are adjacent vertices for all $i=1,\dots,n$. If any two vertices of $X$ can be joined by a path, then $X$ is {\em connected\/}. On any $Y \subset X$, we get the induced graph structure $E|_{Y}=\{\,\{x,y\}\in E\mid x,y\in Y\,\}$. Then $Y \equiv(Y,E|_{Y})$ is called a {\em subgraph\/} of $X$. 

Let $X'\equiv(X',E')$ be another graph. A bijection $X\to X'$ is an {\em isomorphism\/} ({\em of graphs\/}) if it induces a bijection $E\to E'$. Given distinguished points, $x_0\in X$ and $x'_0\in X'$, a ({\em pointed\/}) {\em isomorphism\/} $f:(X,x_0)\to(X',x'_0)$ is an isomorphism $f:X\to X'$ satisfying $f(x_0)=x'_0$. If there is an isomorphism $X\to X'$ (respectively, $(X,x_0)\to(X',x'_0)$), then these structures are called {\em isomorphic\/}, and the notation $X\cong X'$ (respectively, $(X,x_0)\cong(X',x'_0)$) may be used. The composition of isomorphisms is another isomorphism. An isomorphism $X\to X$ (respectively, $(X,x_0)\to(X,x_0)$) is called an {\em automorphism\/} of $X$ (respectively, $(X,x_0)$). The group of automorphisms of $X$ (respectively, $(X,x_0)$) is denoted by $\Aut(X)$ (respectively, $\Aut(X,x_0)$).

For the purposes of this paper, we will only consider connected graphs with finite vertex degrees. For these there is a canonical metric structure $d$, where $d(x,y)$ is the minimum length of paths in $X$ from $x$ to $y$.
For convenience, we will use \emph{disks} or \emph{closed balls} defined with non-strict inequalities, $D (x,r)=\{\,y\in X\mid d(x,y)\le r\,\}$. Also, let $S(x,r)=\{\,y\in X\mid d(x,y)= r\,\}$ denote the \emph{sphere} of center $x$ and radius $r$. We may add the ambient metric space as a subscript ``$D_X(x,r)$" to avoid ambiguity. The following  basic result will be used implicitly throughout the paper.

\begin{lem}
	If $X$ has finite vertex degrees, then its disks are finite\footnote{This means that $X$ is a \emph{proper} metric space, in the sense that its closed balls are compact.} and $X$ is countable.
\end{lem}

\subsection{Colorings}\label{ss: colorings}

Let $\Xi$ be a Cantor space, i.e.\ a compact, metrizable, totally disconnected  topological space with no isolated points. Fix an ultrametric $d_\Xi$ inducing the topology of $\Xi$, which we will denote simply by $d$ when no confusion may arise.
Let $X$ be a  graph. A map  $\phi:X\to \Xi$ is called a ({\em vertex\/}) {\em coloring\/} of $X$, and $(X,\phi)$ is called a {\em colored graph\/}. If $x_0\in Y\subset X$, then the simplified notation $(Y,\phi)=(Y,\phi|_{Y})$ will be used. The following concepts for colored graphs are the obvious extensions of their graph versions: ({\em pointed\/}) {\em isomorphisms\/}, denoted by $f:(X,\phi)\to(X',\phi')$ and $f:(X,x_0,\phi)\to(X',x'_0,\phi')$, {\em isomorphic\/} (pointed) colored graphs, denoted by $(X,\phi)\cong(X',\phi')$ and $(X,x_0,\phi)\cong(X',x'_0,\phi')$, and {\em automorphism\/} groups of (pointed) colored graphs, denoted by $\Aut(X,\phi)$ and $\Aut(X,x_0,\phi)$.

For $R\in \mathbb{N}$, $\epsilon>0$, an \emph{$(R,\epsilon)$-equivalence} $h\colon (X,x,\phi)\rightarrowtail (Y,y,\psi)$ is a pointed graph isomorphism $h\colon (D(x,R),x)\to(D(y,R),y)$ such that $d(\phi(u),\psi(h(u)))<\epsilon$ for every $u\in  D(x,R)$. We use the modified arrow $\rightarrowtail$ to emphasize that $h$ is actually a partial map. It will be said that $(X,x,\phi),(Y,y,\psi)\in \gromov$ are $(R,\epsilon)$-equivalent if there is an $(R,\epsilon)$-equivalence $h\colon (X,x,\phi)\rightarrowtail (Y,y,\psi)$. The following lemma follows immediately from the ultrametric triangle inequality.

\begin{lem}\label{l.composition}
	If $f\colon (X_1,x_1,\phi_1)\rightarrowtail (X_2,x_2,\phi_2)$ and $g\colon (X_2,x_2,\phi_2) \rightarrowtail (X_3,x_3,\phi_3)$ are $(n,\epsilon)$ and $(m,\delta)$-equivalences, respectively, then $g\circ f \colon (X_1,x_1,\phi_1)\rightarrowtail (X_3,x_3,\phi_3)$ is a $(\min\{n,m\},\max\{\epsilon, \delta\})$-equivalencem, where $\dom(g\circ f)=D(x_1,\min\{n,m\})$.
\end{lem}

Let $\widehat\GG_*$ be the set\footnote{These graphs are countable, and therefore we can assume that their underlying sets are contained in $\mathbb{N}$. In this way, $\widehat\GG_*$ becomes a well defined set.} of isomorphism classes, $[X,x,\phi]$, of pointed connected colored graphs, $(X,x,\phi)$, whose vertices have finite degree. For each $R\in\mathbb{Z}^+$, let 
\[
\widehat{\mathcal{N}}_{R,1/R}=\{\,([X,x,\phi],[Y,y,\psi])\in\widehat\GG_*^2\mid [X,x,\phi] \text{ and }[Y,y,\psi] \text{ are $(R,1/R)$-equivalent }\,\}\;.
\]

\begin{rem}
Since the property of being $(R,\epsilon)$-equivalent is clearly invariant by colored graph isomorphisms, it makes sense to say whether two classes are $(R,\epsilon)$-equivalent or not. Thus the subsets $\widehat{\mathcal{N}}_{R,\epsilon}$ are well-defined. In what follows, we will incur in the same abuse of notation for  isomorphism-invariant properties without explicit mention.
\end{rem}

The sets $\widehat{\mathcal{N}}_{R,1/R}$ form a base of entourages of a uniformity on $\widehat\GG_*$, which is easily seen to be complete. Moreover this uniformity is metrizable because this base is countable. A metric inducing this uniformity is given by 
\[
d([X,x,\phi], [Y,y,\psi])= 2^{-n}\;,
\]
where $n$ is the largest integer such that $([X,x,\phi], [Y,y,\psi])\in \widehat{\mathcal{N}}_{n,1/n}$. \Cref{l.composition}  shows that $d$ is actually an ultrametric.

\begin{lem}\label{l.totdis}
    $\gromov$ and $\mathcal{G}_*$ are totally disconnected, perfect and separable metric spaces, hence Polish.
\end{lem}
\begin{proof}
    By the preceeding discussion, we only have to prove separability and perfection. The former follows from the fact that elements $[X,x,\phi]$ with $X$  and $\im \phi $ finite form a countable dense subset. To prove perfection note that we can modify $[X,x,\phi]$ in a region arbitrarily far from $x$, thus obtaining a non-constant sequence converging to $[X,x,\phi]$. The proof for $\mathcal{G}_*$ proceeds similarly.
\end{proof}

By unraveling the definition of the topology of $\gromov$, we obtain the following characterization for almost chaotic colored pointed graphs. It is then clear that, in spite of its formulation, it does not depend on the choice of basepoint.
\begin{lem}\label{l:chaotic}
	Let $(X,x,\phi)$ be a pointed colored graph. Then $(X,\phi)$ is almost chaotic if and only if there is a sequence of  quasi-transitive pointed colored graphs $(Y_n,y_n,\psi_n)$  satisfying:
	\begin{lemenum}
		\item $(X,x,\phi)$ is $(n,1/n)$-equivalent to $( Y _n,y_n, \psi_n)$, and
		\item for all $ n,m\in \mathbb{N}$, there is some  $x_{n,m}\in X$ such that $(X,x_{n,m},\phi)$ is $(m,1/m)$-equivalent to $( Y _n, y_n, \psi_n)$.
	\end{lemenum}
\end{lem}

 Removing the colorings from the notation, we get the Polish space $\GG_*$ of isomorphism classes of pointed connected graphs.
 In this way, we get canonical maps $\ell_{X}:X\to\GG_*$ for connected graphs $X$, defining a canonical partition of $\GG_*$ by the sets $[X]:=\im (\ell_{X})$. Then it is said that $X$ is {\em aperiodic\/} if $\ell_{X}$ is injective,  \emph{quasi-transitive} if $[X]$ is a finite set, \emph{almost chaotic} if the quasi-transitive classes are a dense subset of $\overline{[X]}$, and \emph{chaotic} if $X$ is infinite, aperiodic and almost chaotic. Observe also that the forgetful map $\hat p\colon \widehat\GG_*\to\GG_*$ is continuous. The space $\GG_*$ is a subspace of the Gromov space $\mathcal{M}_*$ of isomorphism classes of pointed proper metric spaces \cite{Gromov1981}, \cite[Chapter~3]{Gromov1999}. The obvious version of Lemma~\ref{l:chaotic}  in this setting follows by considering a constant coloring.

\begin{lem}\label{l.polish}
	The topological spaces $\mathcal{G}_*^\infty$, $\mathcal{G}_*^\infty(\Delta)$,  $\gromov^\infty$ and $\gromov^\infty(\Delta)$ are totally disconnected, separable, perfect and completely metrizable for every choice of graph $X$. Moreover $\mathcal{G}_*^\infty(\Delta)$ and $\gromov^\infty(\Delta)$ are compact and therefore Cantor spaces.
\end{lem}

\begin{proof}
    All the spaces are totally disconnected because $\gromov$ and $\mathcal{G}_*$ are and this is an hereditary property. 
	$\mathcal{G}_*^\infty$, $\mathcal{G}_*^\infty(\Delta)$,  $\gromov^\infty$ and $\gromov^\infty(\Delta)$ are closed subsets of $\mathcal{G}_*$ and $\gromov$, respectively, hence  Polish by \Cref{l.totdis} and~\cite[Prop.~8.1.2]{Cohn}. 
	To prove the compactness of  $\mathcal{G}_*^\infty(\Delta)$ and $\gromov^\infty(\Delta)$ note that, for every $n\in \mathbb{N}$, the disks $(D_X(x,n),x)$ with $[X,x]\in \mathcal{G}_*$  represent only finitely many isomorphism classes. Therefore,  for every sequence $[X_n,x_n]$ in $\mathcal{G}_*^\infty(\Delta)$, we can find a convergent subsequence by a diagonal process. Finally, the compactness of $\Xi$ ensures that, given colorings $\phi_n$, we can find a subsequence $n(k)$ such that $[X_{n(k)},x_{n(k)},\phi_{n(k)}]$ converges in $\gromov^\infty(\Delta)$.
\end{proof}

\section{Proof of Theorem~\ref{t.one}}

Let us start by proving \Cref{i.ac}. Let $\Delta\geq 3$ and fix an arbitrary $\xi\in \Xi$.
Given a pointed colored graph $(X,x,\phi)$ with $X$ infinite, and $n\in \mathbb{N}$, define the pointed colored graph 
\[
(K,k,\psi):=(K,k,\psi)_{X,x,\phi,n}
\]
as follows: Since $X$ is infinite, we can choose $y\in S(x,n)$ such that $\deg_{D(x,n)}y<\deg X$. Let $(\ZZ,\chi)$ be the Cayley graph of the integers with $\chi$ the constant coloring with value $\xi$. Consider a $\ZZ$-indexed family $(D(x_z,n),x_z,\phi_z)_{z\in \ZZ}$ of copies of the pointed colored disk $(D_X(x,n),x,\phi)$, and let $y_z$ denote the corresponding copies of $y$.  Let  $(K,\psi)$ be the  disjoint union of $(\ZZ,\chi)$ and the family $(D(x_z,n),x_z,\phi_z)_{z\in\ZZ}$ with  additional edges joining $y(z)$ to $z$ for each $z\in \ZZ$. Let $k=x_0\subset (D(x_0,n),x_0)$. 

The previous construction depends on the choice of $y\in S(x,n)$. Let us assume that we have a fixed choice of such $y$ for every class   $[X,x,\phi]\in \gromov^\infty$ and $n\in \mathbb{N}$, so that we have a well-defined map $([X,x,\phi],n)\mapsto [K,k,\chi]_{X,x,\phi,n}$. The definition makes the following lemma obvious.

\begin{lem}\label{l.equi}
If $[X,x,\phi]$ is $(m,1/m)$-equivalent to $[Y,y,\psi]$, then $[K,k,\chi]_{X,x,\phi,m}$ is $(m,1/m)$-equivalent to $[K,k,\chi]_{Y,y,\psi,m}$.
\end{lem}

\begin{figure}[tb]\label{f.ch}
	\includegraphics{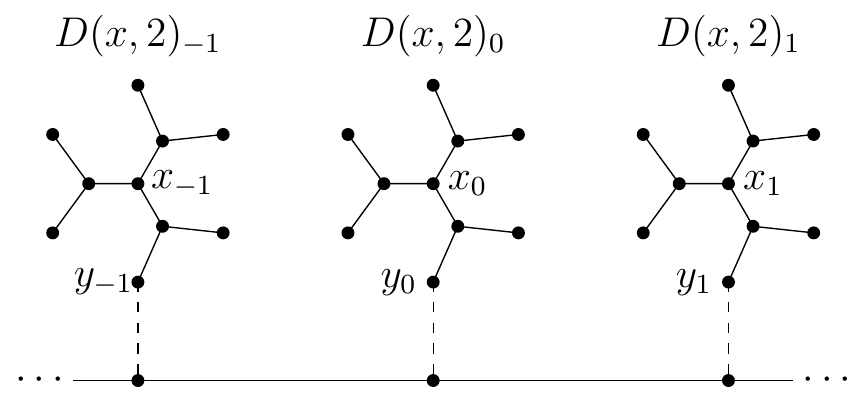}
	\caption{Construction of $K$ from $X$ and $\ZZ$}
\end{figure}

\begin{lem}\label{l.kqt}
For every pointed colored graph $(X,x,\phi)$ with $X$ infinite and $n\in \mathbb{N}$, the colored graph $(K,\psi)_{X,x,\chi,n}$ is quasi-transitive.
\end{lem}
\begin{proof}
For each $m\in \mathbb{N}$, there is an obvious colored graph isomorphism sending each $v\in D(x,n)_z$, $z\in \ZZ$ to the corresponding point in $D(x,n)_{z+m}$, and sending  $z\in \ZZ$ to $z+m\in \ZZ$.
\end{proof}

Let $\widehat{V}(n,r,m)$ be the subset consisting on classes $[X,x,\phi]\in \gromov^\infty$ satisfying that there is some $y\in X$  such that  $d(x,y)\geq r$ and $[X,y,\phi]$ is $(m,1/m)$-equivalent to $[K_{X,x,\phi,n}, k_{X,x,\phi,n}, \psi_{X,x,\phi,n}]$.

\begin{prop}\label{p:v}
	For $n,r,m\in\mathbb{N}$, the sets $\widehat V(n,r,m)$ and $ V(n,r,m)$ are open, dense subsets of $\gromov^\infty$ and $\mathcal{G}_*^\infty$, respectively. For $\Delta\geq 3$, $\widehat V(n,r,m)\cap \gromov^\infty(\Delta)$ and $ V(n,r,m)\cap \mathcal{G}^\infty_*(\Delta)$ are open, dense subsets of $\gromov^\infty(\Delta)$ and  $\mathcal{G}_*^\infty(\Delta)$, respectively.
\end{prop}

\begin{proof}
	We  prove first that $\widehat V(n,r,m)$ is open. Let $[X,x,\phi]\in \widehat V(n,r,m)$, so that there is some $y\in X$ such that $d(x,y)\geq r$ and there is an $(m,1/m)$-equivalence 
	\[
	f\colon (X,y,\phi) \rightarrowtail (K_{X,x,\phi,n}, k_{X,x,\phi,n},\psi_{X,x,\phi,n})\;.
	\]
	Then 
	\[
	\mathcal{N}_{n+m+d(x,y),1/m}[X,x,\phi]\subset \widehat{V}(n,r,m)\;.
	\] 
	Indeed, let $h\colon (X,x,\phi)\rightarrowtail (Z,z,\chi)$ be a $\big(n+m+d(x,y),1/m\big)$-equivalence. By the triangle inequality we have $D(y,n+m)\subset \dom h$, so the restriction $h|_{D(y,n+m)}\colon (X,y,\phi)\rightarrowtail (Z,z,\chi)$ is an $(n+m,1/m)$-equivalence. \Cref{l.composition} yields that 
	\[
	f\circ h^{-1}\colon (Z,z,\chi)\to ( K_{X,x,\phi,n},  k_{X,x,\phi,n}, \psi_{X,x,\phi,n})
	\]
	is an $(m,1/m)$-equivalence. Finally, we get $[Z,z,\chi]\in \widehat V(n,r,m)$ by \Cref{l.equi}.
	
	Let us prove that  $\widehat V(n,r,m)$ is dense in $\gromov^\infty$. Let $[X,x,\phi]\in \gromov$ and $l\in \mathbb{N}$. We will now define a graph $(l,1/l)$-equivalent to $[X,x,\phi]$ and contained in $\widehat V(n,r,m)$. Take the disjoint union of $(D_X(x,\max\{n,l\}),\phi)$, $(D_{ K_n}(k_n,m),\psi)$ and an arbitrarily colored infinite semi-ray $\mathbb{N}$. Then take points $y\in S_X(x,n)$ and $z\in S_{ K_n}( k_n,m)$ and add edges connecting  $y$ and $z$ to $0\in \mathbb{N}$.   It is trivial to check that such a graph satisfies the required conditions. To prove that $\widehat V(n,r,m)$ is dense in $\gromov(\Delta)$, note that $D_X(x,n)$ and $D_{\widetilde K_n}(\tilde k_n,m)$ are disks in infinite graphs of degree $\leq \Delta$. Therefore we can perform the same construction but choosing $y\in S_X(x,n)$ and $z\in S_{\widetilde K_n}(\tilde k_n,m)$ of degree $< \Delta$, and the resulting graph in $\widehat V(n,r,m)$ will have degree $\leq \Delta$.
	
	The same proof applies to $ V(n,r,m)$ by ignoring all references to colorings.
\end{proof}
 
 The following result is an immediate consequence of Lemmas~\ref{l:chaotic} and~\ref{l.kqt}, and \Cref{p:v}. It concludes the proof of \Cref{i.ac}.
 
\begin{prop}
	The sets  $\bigcap \nolimits_{n,r,m\in\mathbb{N}} \widehat V(n,m,r)$  and $\bigcap \nolimits_{n,r,m\in\mathbb{N}}V(n,m,r)$ are generic in $\gromov^\infty$ and $\mathcal{G}_*^\infty$, respectively. Their intersections with $\gromov^\infty(\Delta)$ and $\mathcal{G}_*^\infty(\Delta)$ are also generic. All these sets consist of almost chaotic clases.
\end{prop}

Let us now proceed with the proof of \Cref{i.ap}. Let
\[
W(n)= \{\, [X,x,\phi] \in \gromov^\infty\mid \forall x, y\in D(x,n), \ x\neq y\Rightarrow \phi(x)\neq \phi(y)    \,\}\;.
\]
\begin{prop}\label{p.w}
	$W(n)$ and $W(n)\cap \gromov^\infty(\Delta)$  are  open, dense subsets of $\gromov^\infty$ and $\gromov^\infty(\Delta)$, respectively.
\end{prop}

\begin{proof}
	Let us prove that $W(n)$ is open. Let $[X,x,\phi]\in W(n)$ and
	\[
	\epsilon = \inf \{\,d(\phi(y),\phi(z))\mid y,z\in D(x,n),\ y\neq z \,\}>0\;.
	\]
	Then 
	\[\mathcal{N}_{n,\epsilon/3}[X,x,\phi]\cap \gromov^\infty\subset W(n)\;.
	\]
	Indeed, let $h\colon (Y,y,\psi)\rightarrowtail (X,x,\phi)$ be an $(n,\epsilon/3)$-equivalence, and suppose by absurdity that there are $u,v\in D(y,n)$, $u\neq v$ with $\psi(u)=\psi(v)$. Then we get
	\[
	d(\phi(h(u)),\phi(h(v)))\leq d(\phi(h(u)),\psi(u)) + d(\phi(h(v)),\psi(v)) + d(\psi(u),\psi(v))\leq \frac{2\epsilon}{3}<\epsilon\;, 
	\]
	by the definition of  $(n,\epsilon/3)$-equivalence and the triangle inequality. This clearly contradicts the definition of $\epsilon$.
	  
	To show that $W(n)$ is dense, let $[X,x,\phi]\in W(n)$ and modify the coloring $\phi$ in the following way: for every $y\in D(x,n)$, choose $\hat \phi (y)\in B_\Xi(\phi(y),\epsilon)\setminus \{\phi(y)\}$ so that the restriction of $\hat \phi$ to $D(x,n+r)$ is injective. Clearly this implies $[X,x,\hat \phi]\in W(n)$, and $[X,x,\hat \phi ]$ is $(m,\epsilon)$-equivalent to $[X,x,\phi]$ for every $m\in \mathbb{N}$. Since $[X,x,\hat \phi]\in \gromov^\infty(\Delta)$ if $[X,x,\phi]\in \gromov^\infty(\Delta)$, it follows that   $W(n)\cap \gromov^\infty(\Delta)$ is dense in $\gromov^\infty(\Delta)$.
\end{proof}

\begin{cor}\label{c.w}
	The set $\bigcap \nolimits_{n\in\mathbb{N}} W(n) $ consists of aperiodic classes.
\end{cor}

\begin{proof}
	Let $[X,x,\phi]\in \bigcap \nolimits_{n\in\mathbb{N}} W(n) $, and suppose by absurdity that there is some non-trivial $f\in \Aut(X,\phi)$. Then there are $y,z\in X$, $y\neq z$, such that $f(y)=z$. But this yields $\phi(y)=\phi(z)$, which in turn implies 
	\[
	[X,x,\phi]\notin W(\max\{d(x,y),d(x,z)\})\;,
	\]a contradiction.
\end{proof}

\begin{figure}[tb]
\includegraphics{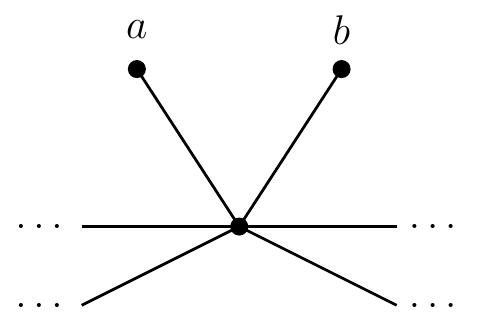}
\caption{In any graph that contains such a pattern, there is a non-trivial isomorphism given by the interchange of $a$ and $b$.}
\label{f.ap}
\end{figure}

This concludes the proof of \Cref{i.ap}. Note that for this result we do not need the assumption $\Delta \geq 3$. The corresponding result is not true if we restrict our attention to non-colored graphs. It is clear that graphs exhibiting the local pattern shown in \Cref{f.ap} cannot be approximated by aperiodic graphs. The same counterexample applies for colored graphs if the space of possible values has an isolated point $\xi$ by using that color to decorate both $a$ and $b$. 

To finish the proof of \Cref{i.ch}, note that for infinite graphs aperiodic implies non-quasi-transitive. Therefore both the set 
\[
\bigcap \limits_{n,m,r\in \mathbb{N}} \widehat{V}(n,m,r)\cap \bigcap \limits_{n\in\mathbb{N}} W(n)
\] 
and its intersection with $\gromov^\infty(\Delta)$, $\Delta\geq 3$, consist of chaotic classes and are generic in $\gromov^\infty$ and $\gromov^\infty(\Delta)$, respectively.

\section{An example of a chaotic colored graph}

The colored graph $[X,\phi]$ defined in \Cref{ex:dense} satisfies $\overline{[X,\phi]}=\gromov^\infty$, so that it being chaotic can be seen as a weakening of \Cref{i.ac}. To prove the density of quasi-transitive classes we take arbitrary finite patterns and embed them in a periodic configuration, as illustrated in Figure~\ref{f.ch}. The same ideas can be used to obtain less trivial examples. In this section we detail the construction of a chaotic graph whose closure is the family of classes $[Y,\psi]$ with every vertex $y\in Y$ satisfying $\deg (y)=1$ or $3$.

Let $C\subset \Xi$ be a  countable, dense subset, and let $\mathcal{F}\subset \gromov$ consist of the classes $[Y,y,\psi]$ such that
\begin{enumerate}[(i)]
    \item \label{i.exap} $[Y,y,\psi]$ is aperiodic,
    \item $Y$ is a finite grpah,
    \item $\im(\psi)\subset C$,
    \item  $\deg(z)=1$ or $3$ for every $z\in Y$, and
    \item \label{i.degy}$\deg (y) = 1$.
\end{enumerate}
Let $F$ be a set of colored graphs that contains exactly one representative for each class in $\mathcal{F}$. Since $F$ is countable, we can choose an enumeration of the form $F=\{(Y_z,y_z,\psi_z)\}_{z\in \ZZ}$. Let $\chi\colon \ZZ\to C$ be an aperiodic coloring, and let $X$ be the colored graph constucted as follows: take the disjoint union of $(\ZZ,\chi)$ and the family $\{ (Y_{z},y_{z},\psi_{z}) \}_{z \in \mathbb{Z}}$, and add an edge joining $y_{z}$ to $z \in \mathbb{Z}$ for every $z\in \ZZ$ (see \Cref{fig:Hn}).  Since  the colored graphs $(Y_{z},y_{z},\psi_{z})$ are aperiodic, the proof that $(X,\phi)$ is aperiodic follows exactly as in \Cref{ex:dense}. 

\begin{figure}[htbp]
 \begin{minipage}{0.50\hsize}
  \begin{center}
   	\includegraphics[width=\columnwidth]{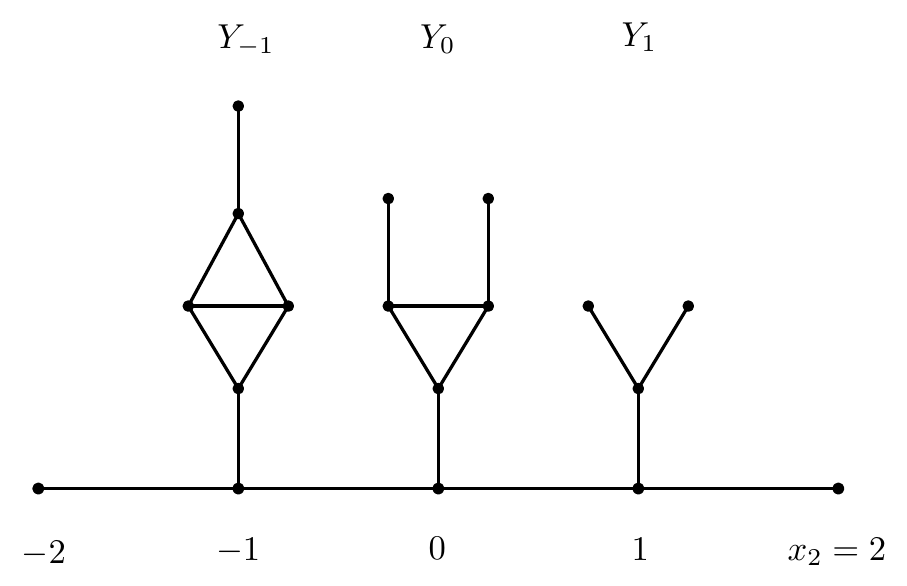}
  \end{center}
  \caption{An example of the graph $H_{2}$}
  \label{fig:Hn}
 \end{minipage}
     \hspace{20pt}
 \begin{minipage}{0.40\hsize}
  \begin{center}
  	\includegraphics[width=\columnwidth]{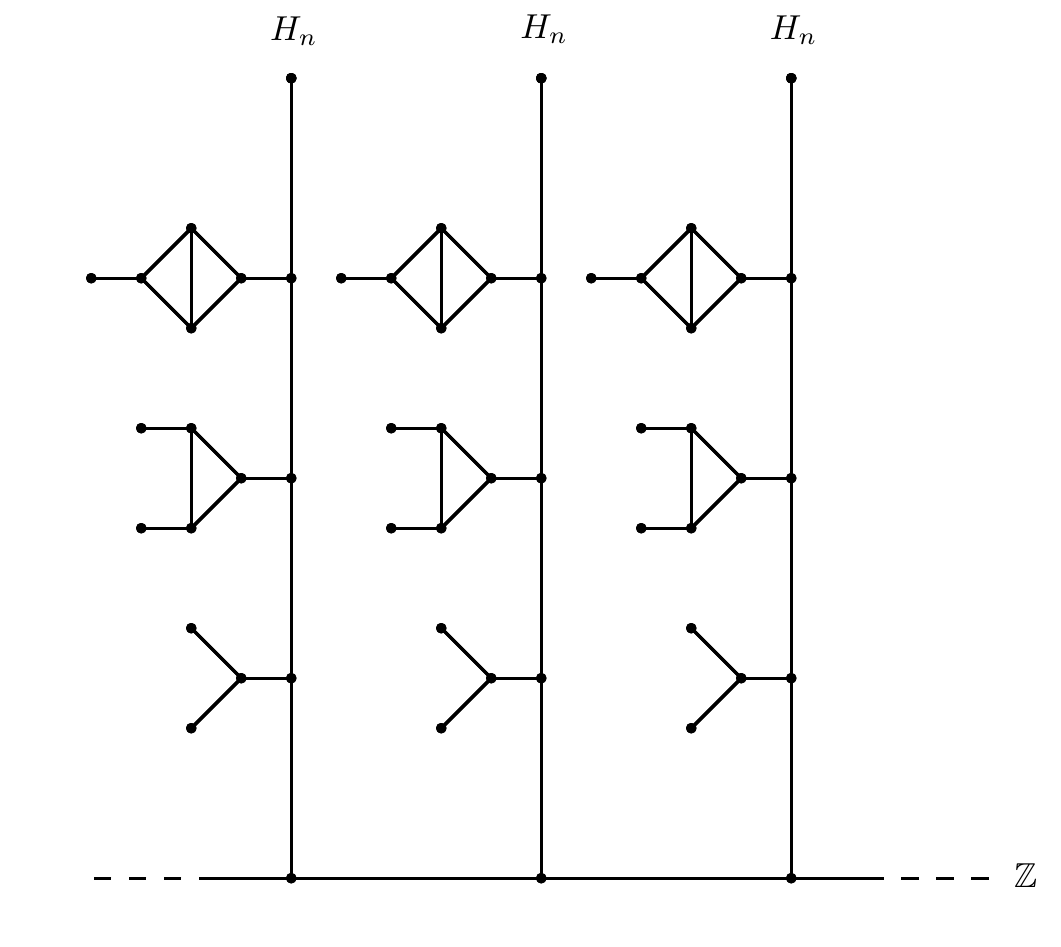}
  \end{center}
  \caption{An example of the graph $Z_{2}$}
  \label{fig:Zn}
 \end{minipage}
\end{figure}

Let us prove that $(X,\phi)$ is almost chaotic. Let $x_0:=0\in \ZZ\subset X$. For $n\in \NN$, let $(H_n,\phi_n)$ be the colored subgraph of $(X,\phi)$ with vertex set 
\[
\{-n,-n+1,\ldots,n\}\cup\bigcup \limits_{-n<z<n } Y_z \;,
\]
where  by $\{-n,-n+1,\ldots,n\}$ we mean the obvious subset of the copy of $\ZZ$ contained inside $X$. Note that $D(x_0,n)\subset H_n$. Let $x_n=n\in H_n$. Then $(H_n,x_n,\phi)$ satisfies conditions~\ref{i.exap}--\ref{i.degy}.
Let us embed $(H_n,\phi)$  into a quasi-transitive graph $(Z_n,\xi_n)$ as follows: the disjoint union of a $\ZZ$-indexed family of copies of $(H_n,\phi)$, with vertex sets denoted by $H_{n,z}$, and a single copy of $\ZZ$, and identify the point corresponding to $x_n$ in $H_{n,z}$ with $z\in \ZZ$ (see \Cref{fig:Zn}). The proof that $(Z_n,\xi_n)$ is quasi-transitive proceeds as in \Cref{l.kqt}. Also $(Z_n,\xi_n)$ contains and isomorphic copy of $(D(x_0,n),\phi)$ by construction. For each $m\in \NN$, the colored subgraph of $(Z_n,\xi_n)$ with vertex set 
\[
\{-m,-m+1,\ldots,m\}\cup\bigcup \limits_{-m<z<m } H_{n,z}
\]
satisfies again conditions~\ref{i.exap}--\ref{i.degy}. This means that we can find an isomorphic copy inside $(X,\phi)$, and thus $[Z_n,\chi_n]\subset \overline{[X,\phi]}$. To recap, we have proved that for each disk $D(x,n)$ we can find a quasi-transitive graph $(Z_n,\chi_n)$ containing a copy of $(D(x,n),\phi)$ and $[Z_n,\chi_n]\subset \overline{[X,\phi]}$. Therefore $(X,\phi)$ is almost chaotic and aperiodic, hence chaotic.
 

\subsection*{Acknowledgements:}
The first author is supported by a ``Canon Foundation in Europe Research Fellowship" and would like to thank the hospitality of Ritsumeikan University. The second author is supported by JSPS KAKENHI Grant Number 17K14195. Both authors were partially supported by FEDER/Ministerio de Ciencia, Innovaci\'{o}n y Universidades/AEI/MTM2017-89686-P.

\end{document}